\documentclass[a4paper,11pt,oneside]{amsart}

\pdfoutput=1

\usepackage[english]{babel}

\usepackage{lmodern}
\usepackage[T1]{fontenc}
\usepackage{microtype}
\usepackage{typearea}

\usepackage{amsmath}
\usepackage{amsthm}
\usepackage{amssymb}
\usepackage{mathtools}
\usepackage{mathrsfs}
\usepackage{dsfont}

\usepackage[autostyle]{csquotes}
\usepackage[all]{xy}

\allowdisplaybreaks
\numberwithin{equation}{section}

\newcommand{\rleft}{\mathopen{}\mathclose\bgroup\left}
\newcommand{\rright}{\aftergroup\egroup\right}

\newtheorem{theorem}{Theorem}[section]
\newtheorem{lemma}[theorem]{Lemma}
\newtheorem{prop}[theorem]{Proposition}

\newtheorem{cor}[theorem]{Corollary}
\newtheorem{conj}[theorem]{Conjecture}
\theoremstyle{definition}
\newtheorem{definition}[theorem]{Definition}
\newtheorem{remark}[theorem]{Remark}
\newtheorem{example}[theorem]{Example}

\numberwithin{equation}{section}

\DeclareMathOperator{\rank}{rank}

\DeclareMathOperator{\vol}{vol}

\newcommand{\st}{\operatorname{str}}
\newcommand{\alg}{\operatorname{alg}}

\newcommand{\Lv}{\lambda}

\newcommand{\Cd}{\mathds{C}}
\newcommand{\Qd}{\mathds{Q}}
\newcommand{\Nd}{\mathds{N}}

\newcommand{\Zd}{\mathds{Z}}

\newcommand{\Pd}{\mathds{P}}

\newcommand{\Ld}{\mathds{L}}
\newcommand{\Ad}{\mathds{A}}

\newcommand{\Z}{\Zd}

\newcommand{\Q}{\Qd}
\newcommand{\C}{\Cd}

\newcommand{\Vm}{\mathcal{V}}

\newcommand{\Fm}{\mathcal{F}}

\newcommand{\Mm}{\mathcal{M}}

\newcommand{\ie}{i.\,e., }

\begin{document}

\title{On the algebraic stringy Euler number}

\author{Victor Batyrev}
\address{Fachbereich Mathematik, Universit\"at
  T\"ubingen, Auf der Morgenstelle 10, 72076 T\"ubingen, Germany}
\email{batyrev@math.uni-tuebingen.de}

\author{Giuliano Gagliardi}
\address{Institut f\"ur Algebra,
  Zahlentheorie und Diskrete Mathematik, Leibniz Universit\"at
  Hannover, Welfengarten 1, 30167 Hannover, Germany}
\email{gagliardi@math.uni-hannover.de}


\begin{abstract}
  We are interested in stringy invariants of singular projective
  algebraic varieties satisfying a strict monotonicity with respect to
  elementary birational modifications in the Mori program. We
  conjecture that the algebraic stringy Euler number is one of such
  invariants. In the present paper, we prove this conjecture for
  varieties having an action of a connected algebraic group $G$ and
  admitting equivariant desingularizations with only finitely many
  $G$-orbits. In particular, we prove our conjecture for arbitrary
  projective spherical varieties.
\end{abstract}

\maketitle

\section{Introduction}

Let $X$ be a smooth projective algebraic variety over $\C$. The
dimension $b^i_{\rm top}(X)$ of the $i$-th cohomology group
$H^{i}(X, \C)$ is called the $i$-th \emph{topological Betti number} of
$X$. The \emph{topological Betti polynomial} of $X$ is defined as
\[ B_{\rm top}(X,t) \coloneqq \sum_{i =0}^{2\dim X} (-1)^i b^i_{\rm
  top}(X) t^i\text{.} \]
Its value
\[ e_{\rm top}(X) \coloneqq B_{\rm top}(X,1) = \sum_{i =0}^{2\dim X}
(-1)^i b^i_{\rm top}(X) \]
is the {\em topological Euler number} of $X$. Let
$H_{\rm alg}^{2i}(X, \C)$ be the $\C$-subspace in $H^{2i}(X, \C)$
generated by the classes $[Z]$ of algebraic cycles $Z \subseteq X$ of
codimension $i$. We call the dimension $b_{\rm alg}^{2i}(X)$ of the
$\C$-space $H_{\rm alg}^{2i}(X, \C)$ the $2i$-th \emph{algebraic Betti
  number} of $X$. This naturally leads us to the \emph{algebraic Betti
  polynomial}
\[ B_{\rm alg}(X,t) \coloneqq \sum_{i =0}^{\dim X} b_{\rm alg}^{2i}(X)
t^{2i}\text{.} \]
Its value
\[ e_{\rm alg}(X) \coloneqq B_{\rm alg}(X,1) = \sum_{i =0}^{\dim X}
b_{\rm alg}^{2i}(X) \]
we call the \emph{algebraic Euler number} of $X$. Using the hard
Lefschetz theorem, one can easily show that
$e_{\rm alg}(X) \geq \dim X + 1$. In particular, the algebraic Euler
number $e_{\rm alg}(X)$ of a smooth projective variety $X$ is always a
positive integer.

In order to consider stringy versions of the above invariants for
singular varieties $X$ we need some additional notation.

Let $X$ be a normal projective algebraic variety over $\C$ and let
$K_X$ be its canonical class. A birational morphism
$\rho \colon Y \to X$ is called a \emph{log-desingularization} of $X$
if $Y$ is smooth and the exceptional locus of $\rho$ consists of
smooth irreducible divisors $D_1, \ldots, D_k$ with simple normal
crossings. Assume that $X$ is $\Q$-Gorenstein, \ie some integral
multiple of $K_X$ is a Cartier divisor on $X$. We set
$I \coloneqq \{ 1, \ldots, k\}$, $D_{\emptyset}\coloneqq Y$, and for
any nonempty subset $J \subseteq I$ we denote by $D_J$ the
intersection of divisors $\bigcap_{j \in J} D_j$, which is either
empty or a smooth projective subvariety in $Y$ of codimension $|J|$.
Then the canonical classes of $X$ and $Y$ are related by the formula
\[ K_Y = \rho^* K_X + \sum_{i=1}^k a_i D_i\text{,} \]
where $a_1, \ldots, a_k$ are rational numbers which are called
\emph{discrepancies}. The singularities of $X$ are said to be
\emph{log-terminal} if $a_i>-1$ for all $i \in I$.

For a projective algebraic variety with at worst $\Q$-Gorenstein
log-terminal singularities one defines the \emph{topological stringy
  Betti function}
\[ B^{\rm str}_{\rm top}(X, t) \coloneqq \sum_{\emptyset \subseteq J
  \subseteq I} B_{\rm top}(D_J, t)\prod_{j \in J}
\left(\frac{t^2-1}{t^{2(a_j+1)} -1} -1 \right) \]
and the \emph{algebraic stringy Betti function}
\[ B^{\rm str}_{\rm alg}(X, t)\coloneqq \sum_{\emptyset \subseteq J
  \subseteq I} B_{\rm alg}(D_J, t)\prod_{j \in J}
\left(\frac{t^2-1}{t^{2(a_j+1)} -1} -1 \right)\text{.} \]
In particular, one obtains the \emph{topological stringy Euler number}
\[ e^{\rm str}_{\rm top}(X) \coloneqq \lim_{t \to 1} B^{\rm str}_{\rm
  top}(X, t) = \sum_{\emptyset \subseteq J \subseteq I} e_{\rm
  top}(D_J)\prod_{j \in J} \left( \frac{-a_j}{a_j+1}\right) \]
and the \emph{algebraic stringy Euler number}
\[ e^{\rm str}_{\rm alg}(X) \coloneqq \lim_{t \to 1} B^{\rm str}_{\rm
  alg}(X, t) = \sum_{\emptyset \subseteq J \subseteq I} e_{\rm
  alg}(D_J)\prod_{j \in J} \left( \frac{-a_j}{a_j+1}
\right)\text{.} \]
It is well-known that the definition of $B^{\rm str}_{\rm top}(X, t)$
and $e^{\rm str}_{\rm top}(X)$ does not depend on the choice of the
log-desingularization $\rho$ (see \cite{Bat-Sing}). The proof of this
statement uses a non-archimedean integration and the fact that the
Betti polynomial $B(X, t)$ can be extended to an additive function on
the category ${\Vm}_\Cd$ of algebraic varieties over $\Cd$ having the
multiplicativity property
\[ B_{\rm top}(X_1 \times X_2,t) = B_{\rm top}(X_1,t) \cdot B_{\rm
  top}(X_2,t)\;\; \forall X_1,X_2 \in {\Vm}_\Cd\text{.} \]

We note that the algebraic Betti polynomial $B_{\rm alg}(X,t)$ can
also be extended to an additive function on the category of arbitrary
algebraic varieties, but the above multiplicativity property fails for
$B_{\rm alg}(X,t)$ in general. Nevertheless, it was noticed in
\cite{teh09} that the weaker multiplicativity property
\[ B_{\rm alg}(X \times \Pd^1,t) =B_{\rm alg}(X,t) \cdot B_{\rm
  alg}(\Pd^1,t)\;\; \forall X \in {\Vm}_\Cd\text{} \]
is sufficient in order to define stringy invariants. In particular,
the algebraic stringy Betti function $B^{\rm str}_{\rm alg}(X, t)$ and
the algebraic stringy Euler number $e^{\rm str}_{\rm alg}(X)$ do not
depend on the choice of the log-desingularization $\rho$.

The algebraic stringy Euler number $e^{\rm str}_{\rm alg}(X)$ of a
singular variety $X$ is usually a rational number. Our first
conjecture claims the positivity of the algebraic stringy Euler
number:

\begin{conj}
  \label{conj1}
  Let $X$ be a projective algebraic variety with at worst
  $\Qd$-Gorenstein log-terminal singularities. Then
  $e^{\rm str}_{\rm alg}(X) >0$.
\end{conj}

\begin{remark}
  It is natural to expect an even stronger inequality:
  \[ e^{\rm str}_{\rm alg}(X) \geq \dim X +1\text{.} \]
\end{remark}

Now we consider some relations between the stringy Euler numbers and
birational geometry.

Recall that two smooth birationally isomorphic projective varieties
$X_1$ and $X_2$ are called $K$-equivalent if there exists a birational
isomorphism $f\colon X_1 \dasharrow X_2$ and a smooth projective
variety $Y$ together with two birational morphisms
$\rho_1\colon Y \to X_1$ and $\rho_2\colon Y \to X_2$ in a commutative
diagram \[
\xymatrix{   & Y\ar[ld]_{\rho_1} \ar[rd]^{\rho_2} & \\
  X_1 \ar@{-->}[rr]^f & & X_2} \]
such that the pullbacks of the canonical classes of $X_1$ and $X_2$
are the same: $\rho_1^*(K_{X_1}) = \rho_2^*(K_{X_2})$. Using methods
of non-archimedean integration, one can show that $K$-equivalent
algebraic varieties must have the same Betti polynomials:
$B_{\rm top}(X_1,t) = B_{\rm top}(X_2,t)$. In particular, two
birationally isomorphic Calabi-Yau manifolds must have the same Betti
numbers (see \cite{Bat-CY}).

The notion of $K$-equivalence immediately extends to algebraic
varieties with at worst $\Q$-Gorenstein singularities. Using the same
method based on non-archimedean integration, one can prove that two
$K$-equivalent projective varieties $X_1$ and $X_2$ with at worst
log-terminal singularities must have the same stringy Betti functions:
$B^{\rm str}_{\rm top}(X_1,t) = B^{\rm str}_{\rm top}(X_2,t)$. In
particular, one has the equality
$e^{\rm str}_{\rm top}(X_1)= e^{\rm str}_{\rm top}(X_2)$ for the
stringy Euler numbers if $X_1$ and $X_2$ are $K$-equivalent.

\begin{definition}
  A proper birational morphism $f\colon X \to X'$ of two
  $\Q$-Gorenstein varieties $X$ and $X'$ is called a \emph{divisorial
    Mori contraction} if $f$ contracts a divisor $D \subseteq X$ and
  the anticanonical divisor $-K_X$ is $f$-ample.
\end{definition}

\begin{definition}
  A birational morphism $g \colon X \dashrightarrow X^+$ of two
  $\Q$-Gorenstein varieties $X$ and $X^+$ together with the birational
  morphisms $f\colon X \to Z$ and $f^+\colon X^+ \to Z$ in a
  commutative diagram \[ \xymatrix{
    X \ar@{-->}[rr]^g\ar[rd]_f&& X^+ \ar[ld]^{f^+}\\
    & Z& } \]
  is called a \emph{Mori flip} if $g$ is an isomorphism in codimension
  one, the anticanonical class $-K_X$ is $f$-ample, and the canonical
  class $K_{X^+}$ is $f^+$-ample.
\end{definition}

We conjecture the following strict monotonicity of the algebraic
stringy Euler number with respect to the above elementary birational
transformations in the Mori program:

\begin{conj}
  \label{conj2a}
  Let $f \colon X \to X'$ be a divisorial Mori contraction. Then
  \[ e^{\rm str}_{\rm alg}(X) > e^{\rm str}_{\rm alg}(X') \text{.} \]
\end{conj}

\begin{conj}
  \label{conj2b}
  Let $g \colon X \dashrightarrow X^+$ be a Mori flip. Then
  \[ e^{\rm str}_{\rm alg}(X) > e^{\rm str}_{\rm alg}(X^+) \text{.}\]
\end{conj}

\begin{remark}
  From the viewpoint of Conjectures \ref{conj1}, \ref{conj2a}, and
  \ref{conj2b}, a projective algebraic variety $X$ with at worst
  terminal singularities is a {\em minimal model} in a given
  birational class if its algebraic stringy Euler number
  $e^{\rm str}_{\rm alg}(X)$ has the minimal possible value among all
  projective algebraic varieties with at worst terminal singularities
  in the same birational class.
\end{remark}

Since $e_{\rm top}(X) = e^{\rm str}_{\rm top}(X)$ for smooth algebraic
varieties, Conjecture~\ref{conj1} is obviously false for the
topological stringy Euler number $e^{\rm str}_{\rm top}(X)$ because
$e_{\rm top}(C) < 0$ already for smooth projective algebraic curves
$C$ of genus $g \geq 2$. If $X$ is a blow-up of a smooth curve $C$ in
a smooth minimal threefold $X'$, then
$e_{\rm top}(X) = e_{\rm top}(X') + e_{\rm top}(C)$ and one gets the
opposite inequality
$e^{\rm str}_{\rm top}(X) < e^{\rm str}_{\rm top}(X')$ if $g \geq 2$.
Therefore Conjecture \ref{conj2a} is also false for topological
stringy Euler numbers. This does not happen for the algebraic Euler
number $e_{\rm alg}(X)$ of smooth projective varieties $X$, because
the cohomology groups in odd degrees do not contribute to it.

In general, the topological and algebraic stringy Euler numbers of a
singular variety are rational numbers. However, if $X =X_\Sigma$ is a
$d$-dimensional $\Q$-Gorenstein projective toric variety associated to
a rational polyhedral fan $\Sigma$ in a $d$-dimensional $\Q$-space
$N_\Q$, then the stringy Euler number
$e^{\rm str}_{\rm top}(X_\Sigma)$ is always a positive integer. This
statement follows from the explicit computation of the stringy Betti
function of the toric variety $X_\Sigma$ by the formula
\[ B^{\rm str}_{\rm top}(X_\Sigma, t) = ( t^2-1)^d \sum_{ n \in N}
t^{-2 \kappa(n)}\text{,} \]
where $N \subseteq N_\Q= N \otimes \Q $ is the lattice of
$1$-parameter subgroups in the $d$-dimensional algebraic torus acting
on $X_\Sigma$ and $\kappa \colon N_\Qd \to \Qd$ is a special piecewise
linear function having the following properties: 1) $\kappa$ is linear
on every cone $\sigma \in \Sigma$; 2) $\kappa$ has value $1$ on
primitive lattice generators of $1$-dimensional cones in $\Sigma$
(see~\cite{Bat-Sing}). Moreover, it was observed in
\cite[Proposition~4.10]{Bat-Sing} that the stringy Euler number
$e^{\rm str}_{\rm top}(X_\Sigma)$ of a $d$-dimensional toric variety
$X_\Sigma$ coincides with the normalized volume
$d!\vol({\rm shed}(\Sigma))$ of the bounded polyhedral set
\[ {\rm shed}(\Sigma)\coloneqq \{ x \in N_\Q \mid \kappa(x) \leq
1\}\text{,}\]
which is called the \emph{shed of the fan} $\Sigma$
(see~\cite[(4.2)]{Reid}). Torus-equivariant birational morphisms of
toric varieties $X_\Sigma \dasharrow X_{\Sigma'}$ can be described by
polyhedral subdivisions of the corresponding fans. It was observed by
Reid in \cite{Reid} that one has strict inclusions
${\rm shed}(\Sigma') \subsetneq {\rm shed}(\Sigma)$ for divisorial
toric Mori contractions $f \colon X_\Sigma \to X_{\Sigma'}$ and
${\rm shed}(\Sigma^+) \subsetneq {\rm shed}(\Sigma)$ for toric Mori
flips $g \colon X_{\Sigma} \dashrightarrow X_{\Sigma^+}$. Therefore,
the statements in Conjectures~\ref{conj2a}~and~\ref{conj2b} hold true
for equivariant birational morphisms in the toric Mori program. It is
not difficult to show that for projective toric varieties $X$ one
always has even the stronger inequality
$e^{\rm str}_{\rm top}(X) \geq \dim X + 1$.

Let $G$ be a connected reductive algebraic group. An irreducible
normal $G$-variety $X$ is called {\em spherical} if a Borel subgroup
$B \subseteq G$ has an open orbit in $X$. Spherical varieties are
generalizations of toric varieties. Following the equivariant Mori
program for toric varieties in \cite{Reid}, the equivariant Mori
program for spherical varieties was considered by Brion and Knop in
\cite{Br93,Br-Kn} (see also the recent papers of Perrin and Pasquier
\cite{Per,Pas15-MMP}). It is known that a $\Q$-Gorenstein spherical
variety is always log-terminal (see \cite{AB04} and
\cite[Proposition~5.6]{Pas15-Sing}). As in the case of toric
varieties, one has the equality
$B^{\rm str}_{\rm alg}(X, t) = B^{\rm str}_{\rm top}(X, t)$ for any
spherical variety $X$. However, the stringy Euler number of a
projective spherical variety $X$ is not always an integer even in the
case when $X$ is Gorenstein (see~\cite[Example 4.6]{Bat-Mor}).

Our main purpose of this paper is to attract attention to the above
conjectures and to prove them for arbitrary projective spherical
varieties:

\begin{theorem}
  \label{spher}
  Let $G$ be a connected reductive algebraic group. Conjectures
  \ref{conj1}, \ref{conj2a} and \ref{conj2b} are true in the
  $G$-equivariant Mori program for projective spherical $G$-varieties.
\end{theorem}

We remark that our proof of the above theorem works even for more
general $G$-varieties having a regular action of a connected algebraic
group $G$. We need only the following well-known properties of
spherical varieties: 1) every spherical $G$-variety $X$ contains only
finitely many $G$-orbits; 2) every singular spherical $G$-variety $X$
always admits a $G$-equivariant log-desingularization $Y$ which is
again a spherical variety; 3) every birational $G$-equivariant map
$g \colon X_1 \dasharrow X_2$ between two spherical $G$-varieties can
be extended to a $G$-equivariant commutative diagram
\[ \xymatrix{   & Y\ar[ld]_{\rho_1} \ar[rd]^{\rho_2} & \\
  X_1 \ar@{-->}[rr]^f & & X_2} \]
for some smooth spherical $G$-variety $Y$ and $G$-equivariant
log-desingularizations $\rho_1 \colon Y \to X_1$ and
$\rho_2 \colon Y \to X_2$.

In Section~\ref{sec:gsi}, we explain the construction of generalized
stringy invariants of algebraic varieties that include the algebraic
stringy Euler number and some other invariants that could attain
minimum on the minimal models in a given birational class.

In Section~\ref{sec:eqf}, we prove Theorem~\ref{spher} as a
consequence of some more general results on the $G$-equivariant
birational Mori theory. We also discuss some generalizations of our
results that include the $G$-equivariant Mori program for pairs
$(X,\Delta)$.

\section{Generalized stringy invariants of algebraic varieties}
\label{sec:gsi}

We denote by $\Vm_\C$ the category of complex algebraic varieties. Let
$K_0(\Vm_\C)$ be the \emph{Grothendieck group of complex algebraic
  varieties}, which is the quotient of the free abelian group
generated by isomorphism classes $[X]$ of complex algebraic varieties
modulo the subgroup generated by $[X] - [X'] - [X \setminus X']$ for
closed subvarieties $X' \subseteq X$. One defines the structure of a
commutative ring on $K_0(\Vm_\C)$ by the equation
$[X] \cdot [X'] \coloneqq [X \times X']$. We can consider the ring
$K_0(\Vm_{\C})$ as a module over the polynomial ring $\Z[\Lv]$ using
the ring homomorphism $\Z[\Lv] \to K_0(\Vm_{\C})$ that maps $\Lv$ to
the class $ \Ld \coloneqq [\Ad^1] \in K_0(\Vm_{\C})$.

Let $S \coloneqq S(\lambda) \subseteq \Z[\Lv]$ be the multiplicative
subset generated by $\Lv$ and the polynomials $\sum_{i=0}^k \Lv^i$ for
all $k \geq 0$. We define a filtration $F^*_\infty$ on the localized
ring $S^{-1} \Z[\Lv]$ using the subgroups $F^m_\infty S^{-1} \Z[\Lv]$
$(m\in \Z)$ generated by all rational functions $p(\Lv)/q(\Lv)$ such
that $q(\Lv) \in S$ and $\deg q - \deg p \geq m$. This filtration is
obtained by the restriction to $S^{-1} \Z[\Lv]$ of the filtration on
the ring of rational functions on $\Pd^1$ over ${\rm Spec}\, \Z$
defined by the discrete valuation corresponding to the point
$\infty \in \Pd^1$.

We define the filtration ${\mathcal F}^*$ on the
$S^{-1}\Z[\Lv]$-algebra
\[\widetilde{\Mm}_\C \coloneqq S^{-1}K_0(\Vm_{\C})\]
where ${\mathcal F}^m \widetilde{\Mm}_\C$ is the subgroup of
$\widetilde{\Mm}_\C$ generated by
\begin{align*}
  \rleft\{ \frac{[X]}{
  q(\Ld)} : \text{$\deg q(\Lv) -
  \dim X \ge m$, $q(\Lv) \in S$} \rright\}\text{.}
\end{align*}
This filtration is compatible with the filtration $F^*_\infty$ on
$S^{-1}\Z[\Lv]$:
\[ F^m_\infty \cdot {\mathcal F}^l \subseteq {\mathcal F}^{m+l} \;\;
\forall m, l \in \Z. \]

In papers devoted to motivic integration (see, for instance,
\cite[(3.2)]{dl99} or \cite[Definition~2.11]{cra04}) one often
considers the ring $\smash{\widehat{\Mm}_\C}$ obtained as the
completion of the ring $\widetilde{\Mm}_\C$ with respect to the
filtration $\Fm^*$. However, the kernel of the completion map
$\smash{\widetilde{\Mm}_\C} \to \smash{\widehat{\Mm}_\C}$ is rather
difficult to understand because $\widetilde{\Mm}_\C$ is not a
noetherian ring. The above filtration is necessary for the definition
of the topology and the convergence in this topology, but we do not
really need the completed ring $\widehat{\Mm}_\C$ because all
non-archimedean motivic integrals considered will take their values in
the localized ring $\smash{\widetilde{\Mm}_\C}$ or in an extension
$\smash{\widetilde{\Mm}_\C^{(r)}}$, which we explain next.

Let $r$ be a positive integer. Using the ring monomorphism
\begin{align*}
  \imath_r \colon S^{-1} \Z[\Lv] \to S^{-1} \Z[\Lv], \;\;
  \Lv \mapsto \Lv^r\text{,}
\end{align*}
we can identify the image of $\imath_r$ with $S^{-1} \Z[\Lv]$ and
obtain the ring extension
\[S^{-1} \Z[\Lv] \subseteq S(\Lv^{1/r})^{-1}\Z[\Lv^{1/r}]\text{.}\]
This allows us to consider any rational function
\[ \frac{\Lv -1 }{\Lv^{k/r} -1}, \;\; k \in \Nd, \]
as an element of $S(\Lv^{1/r})^{-1}\Z[\Lv^{1/r}]$. For any
$\Z[\Lv]$-module $M$, we define
\[ S^{-1}_{(r)}M \coloneqq S(\Lv^{1/r})^{-1}\Z[\Lv^{1/r}]
\otimes_{S^{-1}\Z[\Lv]} S^{-1}M\text{,} \] and we define
\[\widetilde{\Mm}_\C^{(r)}\coloneqq
S^{-1}_{(r)}K_0(\Vm_{\C})\text{.}\]

The following Definition~\ref{def:str} and Theorem~\ref{th:str} are
slightly modified versions of some of the content from
\cite[Section~4]{teh09}.

\begin{definition}
  \label{def:str}
  Let $X$ be a normal projective variety over $\Cd$ with at worst
  $\Qd$-Gorenstein log-terminal singularities. Denote by $r$ the
  minimal positive integer such that $rK_X$ is a Cartier divisor. Let
  $\rho\colon Y \to X$ be a log-desingularization together with smooth
  irreducible divisors $D_1, \dots, D_k$ with simple normal crossings
  whose support covers the exceptional locus of $\rho$. We can
  uniquely write
  \begin{align*}
    K_Y = \rho^*K_X + \sum_{i=1}^k a_iD_i
  \end{align*}
  for some rational numbers $a_i \in \frac{1}{r}\Z$ satisfying the
  additional condition $a_i = 0$ if $D_i$ is not in the exceptional
  locus of $\rho$. We set $I \coloneqq \{ 1, \ldots, k\}$ and, for any
  $\emptyset \subseteq J \subseteq I$, we define
  \begin{align*}
    D_J \coloneqq
    \begin{cases}
      Y &\text{if $J = \emptyset$,}\\
      \bigcap_{j \in J} D_j &\text{if $J \ne \emptyset$,}
    \end{cases}
    && D_J^\circ \coloneqq D_J \setminus
    \bigcup_{j \in I \setminus J} D_j\text{.}
  \end{align*}

  Let $\phi\colon K_0(\Vm_{\C}) \to M$ be a $\Zd[\Lv]$-module
  homomorphism. It induces a $S^{-1}_{(r)}\Z[\Lv]$-module homomorphism
  \begin{align*}
    S^{-1}_{(r)}\phi \colon  \widetilde{\Mm}_\C^{(r)} \to  S^{-1}_{(r)}M\text{.}
  \end{align*}
  We define the \emph{generalized stringy invariant} associated to $\phi$ as
  \begin{align*}
    \phi^{\st}(X) &\coloneqq
    \sum_{\emptyset \subseteq J \subseteq I}
    \rleft(\prod_{j \in J}\frac{\Lv-1}{\Lv^{a_j+1}-1}\rright)
    \cdot \phi(D_J^\circ)\\
    &=
    \sum_{\emptyset \subseteq J \subseteq I}
    \rleft(\prod_{j \in J}\frac{\Lv-1}{\Lv^{a_j+1}-1}-1\rright) \cdot \phi(D_J)
    \in S^{-1}_{(r)}M
  \end{align*}
\end{definition}
\begin{theorem}
  \label{th:str}
  The definition of $\phi^{\st}(X)$ is independent of the choice of
  the log-de\-sin\-gulari\-za\-tion.
\end{theorem}
\begin{proof}
  It follows from the theory of motivic integration (see, for
  instance, \cite[Appendix]{vey01}) that the element
  \begin{align*}
    [X]^{\st} \coloneqq \sum_{\emptyset \subseteq J \subseteq I}
    \rleft(\prod_{j \in J}\frac{\Ld-1}{\Ld^{a_j+1}-1}\rright) \cdot [D_J^\circ]
    \in \widetilde{\Mm}_\C^{(r)}
  \end{align*}
  is independent of the choice of the log-desingularization. By
  $\phi^{\st}(X) = S^{-1}_{(r)}\phi([X]^{\st})$, the element
  $\phi^{\st}(X)$ is also independent of the choice of the
  log-desingularization.
\end{proof}

We now consider examples of $\Zd[\Lv]$-module homomorphisms
$\phi \colon K_0(\Vm_{\C}) \to M$ and their respective stringy
invariants.

\begin{example}
  \label{ex:1}
  The \emph{virtual Hodge polynomial} (also known as the \emph{$E$-polynomial})
  \begin{align*}
    E(X; u,v) \coloneqq \sum_{p,q} (-1)^{p+q} h^{p,q}(X) u^pv^q
  \end{align*}
  induces a $\Zd[\Lv]$-module homomorphism
  $K_0(\Vm_{\C}) \to \Zd[u,v]$ where the $\Zd[\Lv]$-module structure
  on $\Zd[u,v]$ is given by the map $\Lv \mapsto uv$. By
  Theorem~\ref{th:str}, we obtain the \emph{stringy $E$-function}
  \[ E^{\st}(X; u,v) \in S^{-1}_{(r)}\Zd[u, v]\text{.} \]
\end{example}

\begin{example}
  The topological virtual Betti polynomial
  \begin{align*}
    B_{\rm top}(X; t) \coloneqq E(X; t, t)
  \end{align*}
  induces a $\Zd[\Lv]$-module homomorphism $K_0(\Vm_{\C}) \to \Zd[t]$
  where the $\Zd[\Lv]$-module structure on $\Zd[t]$ is given by the
  map $\Lv \mapsto t^2$. By Theorem~\ref{th:str}, we obtain the
  \emph{topological stringy Betti function}
  \[ B^{\st}_{\rm top}(X; t) \in S^{-1}_{(r)}\Zd[t^2]\text{.} \]
\end{example}

\begin{example}
  The topological Euler number
  \begin{align*}
    e_{\rm top}(X) \coloneqq B(X; 1) = E(X; 1, 1)
  \end{align*}
  induces a $\Zd[\Lv]$-module homomorphism $K_0(\Vm_{\C}) \to \Zd$
  where the $\Zd[\Lv]$-module structure on $\Zd$ is given by the map
  $\Lv \mapsto 1$. By Theorem~\ref{th:str}, we obtain the
  \emph{topological stringy Euler number}
  \[ e^{\st}_{\rm top}(X) \in S^{-1}_{(r)}\Zd\text{.} \]
  Using the map $S^{-1}_{(r)}\Zd \to \Qd$, $\lambda^{1/r} \mapsto 1$,
  we interpret $e^{\st}_{\rm top}(X) \in \Qd$, and then we have
  \[ e^{\st}_{\rm top}(X) = \lim_{t\to 1} B^{\st}_{\rm top}(X; t) =
  \lim_{u,v\to 1} E^{\st}(X; u,v)\text{.} \]
\end{example}

The following was shown by Teh using \cite[Theorem~3.1]{bit04}.

\begin{prop}[{\cite[Proposition~4, $k=2p$]{teh09}}]
  The algebraic Betti polynomial of smooth projective algebraic
  varieties
  \begin{align*}
    B_{\alg}(X;t) \coloneqq \sum_{i} b_{\alg}^{2i}(X) t^{2i}
  \end{align*}
  can be extended to a $\Z[\Lv]$-module homomorphism
  \[ B_{\alg}\colon  K_0(\Vm_{\C}) \to \Z[t^2] \]
  where the $\Zd[\Lv]$-module structure on $\Zd[t^2]$
  is given by $\Lv \mapsto t^2$.
\end{prop}

\begin{cor}
  By Theorem~\ref{th:str}, we obtain
  that the algebraic stringy Betti function
  \begin{align*}
    B^{\st}_{\alg}(X; t) \in S^{-1}_{(r)} \Z[t]
  \end{align*}
  does not depend on the choice of the log-desingularization and the
  same is true for the algebraic stringy Euler number
  \begin{align*}
    e^{\st}_{\alg}(X) = \lim_{t\to 1} B^{\st}_{\alg}(X; t)\in \Qd\text{.}
  \end{align*}
\end{cor}

\begin{example}
  It is also possible to consider the following modified version of
  the virtual $E$-polynomial:
  \begin{align*}
    E_{\rm even}(X; u,v) \coloneqq \frac{1}{2} \left( E(X; u,v) +
    E(X; -u,-v) \right) =   \sum_{p+q \in 2\Z}  h^{p,q}(X) u^pv^q\text{.}
  \end{align*}
  The corresponding stringy invariant
  \[ E_{\rm even}^{\st}(X; u,v) \coloneqq \sum_{\emptyset \subseteq J
    \subseteq I} E_{\rm even}(D_J^\circ; u,v) \rleft(\prod_{j \in
    J}\frac{uv-1}{(uv)^{a_j+1}-1}\rright) \]
  we call the \emph{even stringy $E$-function}.
\end{example}

\begin{example}
  Another modified version of the virtual $E$-polynomial is
  \begin{align*}
    E_{p=q}(X; u,v) \coloneqq
    \sum_{0 \leq p \leq \dim X}  h^{p,p}(X) u^pv^p\text{.}
  \end{align*}
  This leads to the stringy invariant
  \[ E_{p=q}^{\st}(X; u,v) \coloneqq \sum_{\emptyset \subseteq J
    \subseteq I} E_{p=q}(D_J^\circ; u,v) \rleft(\prod_{j \in
    J}\frac{uv-1}{(uv)^{a_j+1}-1}\rright)\text{.} \]
\end{example}

\begin{remark}
  We also obtain the modified Euler numbers
  \begin{align*}
    e^{\st}_{\rm even}(X) \coloneqq \lim_{u,v \to 1}
    E^{\st}_{\rm even}(X; u,v)
    && \text{and} &&
    e^{\st}_{p=q}(X) \coloneqq \lim_{u,v \to 1}
    E^{\st}_{p=q}(X; u,v)\text{.}
  \end{align*}
  The results of the next section show that one can expect that
  Conjectures \ref{conj1}, \ref{conj2a}, and \ref{conj2b} could also
  be true for the rational numbers $e^{\st}_{\rm even}$ and
  $e^{\st}_{p=q}$, \ie these alternative stringy invariants are also
  candidates for invariants attaining their minimum on the minimal
  model in a given birational class.
\end{remark}

\section{Equivariant desingularizations with finitely many orbits}
\label{sec:eqf}

Let $G$ be a connected linear algebraic group. We consider the
category of algebraic $G$-varieties with $G$-equivariant birational
maps over $\Cd$. We need two results of Brion and Peyre.

\begin{lemma}[{\cite[Lemma~1 and Lemma~2]{bp02}}]
  \label{lemma:bp1}
  Let $H \subseteq G$ be a closed subgroup. Then there exists a
  locally trivial fibration in the Zariski topology $G/H \to Z$ with
  fiber $T$ such that $B_{\alg}(Z; t) = B_{\rm top}(Z; t)$ and $T$ is
  an algebraic torus.
\end{lemma}

\begin{prop}
  Let $H \subseteq G$ be a closed subgroup. Then for the
  quasi-projective variety $G/H$ one has
  \begin{equation*}
    B_{\alg}(G/H; t) = B_{\rm top}(G/H; t)\text{.}
  \end{equation*}
  In particular, one has $e_{\alg}(G/H) = e_{\rm top}(G/H)$.
\end{prop}
\begin{proof}
  Using the fibration from Lemma~\ref{lemma:bp1}, one obtains
  \begin{align*}
    B_{\alg}(G/H; t) =  B_{\alg}(T; t) B_{\alg}(Z; t)
    = B_{\rm top}(T; t) B_{\rm top}(Z; t) = B_{\rm top}(G/H; t)
  \end{align*}
  because $B_{\alg}(T; t) = B_{\rm top}(T; t) = (t^2-1)^{\dim T}$.
\end{proof}

\begin{prop}[{\cite[Theorem~1(b)]{bp02}}]
  \label{prop:egh}
  Let $H \subseteq G$ be a closed subgroup. Then one has
  \begin{align*}
    e_{\rm top}(G/H) =
    \begin{cases}
      0&\text{if $\rank G > \rank H$,} \\
      |W_G|/|W_H| &\text{otherwise,}
    \end{cases}
  \end{align*}
  where $W_G$ (resp.~$W_H$) denotes the Weyl group of $G^{\rm red}$
  (resp.~of $(H^\circ)^{\rm red}$). In particular, one has
  $e_{\rm top}(G/H) > 0$ if $H$ is a parabolic subgroup.
\end{prop}

\begin{theorem}
  \label{th:fo-0}
  Let $X$ be a $\Qd$-Gorenstein projective $G$-variety with at worst
  log-terminal singularities. Assume that there exists a
  $G$-equivariant log-desingularization $\rho\colon Y \to X$ such that
  $Y$ contains only finitely many $G$-orbits. Then
  $e_{\alg}^{\st}(X) > 0$.
\end{theorem}
\begin{proof}
  We write
  \begin{align*}
    K_Y = \rho^*K_{X} + \sum_{i=1}^k a_iD_i
  \end{align*}
  for some $G$-invariant smooth irreducible divisors $D_1,\dots,D_k$.
  By Definition~\ref{def:str}, we have
  \begin{align*}
    e_{\alg}^{\st}(X) &= \sum_{\emptyset \subseteq J \subseteq I}
    e_{\alg}(D_J^\circ)\prod_{j \in J}\frac{1}{a_j+1}\text{.}
  \end{align*}
  Since the $G$-invariant locally closed subvariety $D^\circ_J$ is a
  union of finitely many $G$-orbits, by Proposition~\ref{prop:egh}, we
  obtain $e_{\alg}(D_J^\circ) \ge 0$ for all $J \subseteq I$. Together
  with the inequalities $a_j +1 >0$ for all $j \in I$ it implies
  $e_{\alg}^{\st}(X) \geq 0$. Choose a subset $J \subseteq I$ with
  $D_J^\circ = D_J \ne \emptyset$. As $D_J^\circ$ contains a
  projective $G$-orbit, by Proposition~\ref{prop:egh}, we obtain
  $e_{\alg}(D_J^\circ) > 0$ and, hence, the strict inequality
  $e_{\alg}^{\st}(X) > 0$.
\end{proof}

\begin{theorem}
  \label{th:fo-flip}
  Let $g \colon X \dashrightarrow X^+$ be a $G$-equivariant Mori flip
  with $G$-equivariant log-desingularizations $\rho \colon Y \to X$
  and $\rho_{+} \colon Y \to X^+$ as well as a $G$-equivariant
  commutative diagram \[ \xymatrix{   & Y\ar[ld]_{\rho} \ar[rd]^{\rho_+} & \\
    X \ar@{-->}[rr]^g & & X^+} \]
  such that $Y$ contains only finitely many $G$-orbits. Then we have
  $e_{\alg}^{\st}(X) > e_{\alg}^{\st}(X^+)$.
\end{theorem}
\begin{proof}
  Write
  \begin{align*}
    K_Y = \rho^*K_{X} + \sum_{i=1}^k a_iD_i
    = \rho_+^*K_{X^+} + \sum_{i=1}^k a_i^+D_i
  \end{align*}
  for some $G$-invariant smooth irreducible divisors $D_1,\dots,D_k$.
  By Definition~\ref{def:str}, we have
  \begin{align*}
    e_{\alg}^{\st}(X) &= \sum_{\emptyset \subseteq J \subseteq I}
    e_{\alg}(D_J^\circ)\prod_{j \in J}\frac{1}{a_j+1}\text{,}\\
    e_{\alg}^{\st}(X^+) &= \sum_{\emptyset \subseteq J \subseteq I}
    e_{\alg}(D_J^\circ)\prod_{j \in J}\frac{1}{a^+_j+1}\text{.}
  \end{align*}
  According to \cite[Lemma~3.4]{is05}, we have $a_j \le a^+_j$ for
  every $j \in I$ and there exists $j_0 \in I$ such that
  $a_{j_0} < a^+_{j_0}$. Since the $G$-invariant locally closed
  subvariety $D^\circ_J$ is a union of finitely many $G$-orbits, by
  Proposition~\ref{prop:egh}, we have $e_{\alg}(D_J^\circ) \ge 0$ for
  all subsets $J \subseteq I$. Thus, we obtain
  $e_{\alg}^{\st}(X) \ge e_{\alg}^{\st}(X^+)$. Choose a subset
  $J \subseteq I$ with $j_0 \in J$ and
  $D_J^\circ = D_J \ne \emptyset$. As $D_J^\circ$ contains a
  projective $G$-orbit, by Proposition~\ref{prop:egh}, we obtain
  $e_{\alg}(D_J^\circ) > 0$ and, hence, the strict inequality
  $e_{\alg}^{\st}(X^-) > e_{\alg}^{\st}(X^+)$.
\end{proof}

\begin{theorem}
  \label{th:fo-divc}
  Let $f\colon X \to X'$ be a $G$-equivariant divisorial Mori
  contraction with a $G$-equivariant log-desingularization
  $\rho \colon Y \to X$ such that $Y$ contains only finitely many
  $G$-orbits. Then $e^{\alg}_{\st}(X) > e^{\alg}_{\st}(X')$.
\end{theorem}
\begin{proof}
  We write
  \begin{align*}
    K_Y = \rho^*(K_{X}) + \sum_{i=1}^k a_iD_i
    = \rho^*f^*K_{X'} + \sum_{i=1}^k a'_iD_i
  \end{align*}
  for some $G$-invariant smooth irreducible divisors $D_1,\dots,D_k$.
  By Definition~\ref{def:str}, we have
  \begin{align*}
    e_{\alg}^{\st}(X) &= \sum_{\emptyset \subseteq J \subseteq I}
    e_{\alg}(D_J^\circ)\prod_{j \in J}\frac{1}{a_j+1}\text{,}\\
    e_{\alg}^{\st}(X') &= \sum_{\emptyset \subseteq J \subseteq I}
    e_{\alg}(D_J^\circ)\prod_{j \in J}\frac{1}{a'_j+1}\text{.}
  \end{align*}
  Again we have $a_j \le a'_j$ for every $j \in I$ and there exists
  $j_0 \in I$ such that $a_{j_0} < a'_{j_0}$. The rest of the proof is
  the same as in the proof of Theorem~\ref{th:fo-flip}.
\end{proof}

\begin{remark}[Proof of Theorem~\ref{spher}]
  Spherical varieties have only finitely many $G$-orbits (see, for
  instance, \cite[Theorem~2.1.2]{Per}). Moreover, the
  desingularizations of spherical varieties required in
  Theorems~\ref{th:fo-0}, \ref{th:fo-flip}, and $\ref{th:fo-divc}$ can
  be always constructed explicitly using subdivisions of the
  corresponding colored fans (see, for instance, \cite[Theorem~3.1.13
  and Corollary~3.3.8]{Per}).
\end{remark}

\begin{remark}
  As in \cite{bat99}, it is possible to generalize
  Definition~\ref{def:str} to define stringy invariants for any
  Kawamata log-terminal pair $(X, \Delta)$.

  For example, let $\rho\colon Y \to X$ be a log-desingularization of
  $(X, \Delta)$ together with smooth irreducible divisors
  $D_1, \dots, D_k$ with simple normal crossings whose support covers
  the exceptional locus of $\rho$ and the support of $\rho^*\Delta$.
  We can uniquely write
  \begin{align*}
    K_Y = \rho^*(K_X + \Delta) + \sum_{i=1}^k a_iD_i
  \end{align*}
  for some rational numbers $a_i \in \frac{1}{r}\Z$ satisfying the
  additional condition that $-a_i$ is the multiplicity of $D_i$ in
  $\Delta$ if $D_i$ is not in the exceptional locus of $\rho$. We then
  define
  \begin{align*}
    e_{\alg}^{\st}(X, \Delta) \coloneqq
    \sum_{\emptyset \subseteq J \subseteq I}
    e_{\rm alg}^{\st}(D_j^\circ) \prod_{j \in J}\frac{1}{a_j+1} \in \Qd\text{.}
  \end{align*}
\end{remark}

\begin{conj}
  The algebraic stringy Euler number $e^{\rm str}_{\rm alg}(X,\Delta)$
  is strictly monotone with respect to the elementary birational
  transformations from the Mori theory of Kawamata log-terminal pairs
  $(X, \Delta)$.
\end{conj}

It is possible to prove generalized versions of
Theorems~\ref{th:fo-0}, \ref{th:fo-flip}, and $\ref{th:fo-divc}$ for
Kawamata log-terminal pairs $(X, \Delta)$ where the irreducible
components of $\Delta$ are $G$-invariant. In particular, we obtain the
following result.

\begin{theorem}
  The algebraic stringy Euler number $e^{\rm str}_{\rm alg}(X,\Delta)$
  is strictly monotone with respect to the elementary $G$-equivariant
  birational transformations from the $G$-equivariant Mori theory of
  Kawamata log-terminal pairs $(X, \Delta)$ for projective spherical
  varieties $X$ and $G$-invariant divisors $\Delta$.
\end{theorem}

\begin{remark}
  Let $B\subseteq G$ be a Borel subgroup. Since every divisor on a
  spherical variety is linearly equivalent to a $B$-invariant divisor,
  it is natural to consider the $B$-equivariant Mori theory of pairs
  $(X, \Delta)$ on spherical $G$-varieties $X$. Irreducible
  $B$-invariant divisors on $X$ that are not $G$-invariant are called
  \emph{colors}. As in the case of toric varieties, $B$-invariant
  Cartier divisors can be described by piecewise linear functions on
  the colored fan of a spherical variety. It is possible that the
  irreducible $B$-invariant divisors (colors) have singularities, even
  in the open $G$-orbit, and even if all colors are smooth, they may
  not satisfy the normal crossing condition. The $B$-equivariant
  log-desingularizations of $B$-invariant divisors on horospherical
  varieties have been constructed by Pasquier in
  \cite{Pas15-Sing_horo} using Bott-Samelson resolutions of Schubert
  cells in homogeneous spaces. We remark that every spherical variety
  $X$ contains only finitely many $B$-orbits. However, a
  $B$-equivariant log-desingularization of $B$-invariant pairs on a
  spherical variety $X$ having only finitely many $B$-orbits may not
  exist in general.
\end{remark}

\begin{prop}
  Consider the spherical $\operatorname{SL}(2)$-variety
  $X \coloneqq \Pd(\mathfrak{sl}_2) \cong \Pd^2$, \ie the
  projectivization of the adjoint representation of
  $\operatorname{SL}(2)$. Then the set of $B$-invariant divisors in
  $X$ does not admit a $B$-equivariant log-desingularization with
  finitely many $B$-orbits.
\end{prop}

\begin{proof}
  It is easy to show that $X$ contains exactly two irreducible
  $B$-invariant divisors $D_0$ and $D_1$. The divisor
  $D_0 \cong \Pd^1$ is an $\operatorname{SL}(2)$-invariant conic in
  $\Pd^2$. It is the projectivization of the set of nilpotent matrices
  in $\mathfrak{sl}_2$. The $B$-invariant divisor $D_1 \cong \Pd^1$ is
  the tangent line to the conic $D_0$ at the $B$-fixed point
  $x_B \in D_0$.

  Assume that there exists a $B$-equivariant desingularization
  $\rho\colon Y \to X$ such that $Y$ has only finitely many $B$-orbits
  and the set of $B$-invariant divisors in $Y$ has simple normal
  crossings. Then the birational morphism $\rho\colon Y \to X$ between
  two smooth projective algebraic surfaces $Y$ and $X$ can be
  decomposed into a sequence of $B$-equivariant blow-downs
  $ Y \eqqcolon X_n \to X_{n-1} \to \cdots \to X_2 \to X_1 \to X_0
  \coloneqq X$.
  Moreover, every smooth projective algebraic surface $X_i$ must have
  finitely many $B$-orbits. Since $X_0$ has only one $B$-fixed point
  $x_0 \coloneqq D_0 \cap D_1$, the birational morphism
  $\rho_1\colon X_1 \to X_0$ must be the blow-up of this point. Then
  the $B$-variety $X_1$ contains three $B$-invariant divisors
  $\rho^{-1}(D_0)$, $\rho^{-1}(D_1)$, and $\rho^{-1}(x_0)$ having a
  common $B$-fixed point $x_1 \in X_1$. It is easy to show that $x_1$
  is the unique $B$-fixed point in $X_1$. Therefore, the birational
  $B$-equivariant morphism $\rho_2\colon X_2 \to X_1$ must be the
  blow-up of $x_1$. On the other hand, the tangent spaces of three
  divisors at the point $x_1$ are pairwise different. Hence, the group
  $B$ has at least $3$ pairwise distinct fixed points in the
  exceptional divisor $E \coloneqq \rho^{-1}(x_1) \subseteq X_2$. So
  $B$ must act trivially on $E$, \ie there exist infinitely many
  $B$-orbits in $X_2$, a contradiction.
\end{proof}

\bibliographystyle{amsalpha}
\bibliography{otasen}

\end{document}